\documentclass[12pt]{amsart}
\usepackage{verbatim,amsmath,latexsym,amssymb,amsbsy}
\usepackage{diagrams}

\theoremstyle{plain}
\newtheorem{thm}{Theorem}[section]

\newtheorem{lem}[thm]{Lemma}

\newtheorem{cor}[thm]{Corollary}

\theoremstyle{definition}

\newtheorem{df}[thm]{Definition}

\newtheorem{ex}[thm]{Example}
\newtheorem{ex-notn}[thm]{Example/Notation}

\newtheorem{conj}[thm]{Conjecture}

\newtheorem{rem}[thm]{Remark}

\def\im{\operatorname{im}}

\def\CC{{\mathbb C}}     
\def\HH{{\mathbb H}}

\def\PP{{\mathbb P}}     
\def\QQ{{\mathbb Q}}    
 
\def\ZZ{{\mathbb Z}}

\def\D{{\Delta}}

\def\calA{{\mathcal A}}
\def\calC{{\mathcal C}}

\def\calE{{\mathcal E}}

\def\calH{{\mathcal H}}
\def\calI{{\mathcal I}}

\def\calO{{\mathcal O}}

\def\calV{{\mathcal V}}

\def\bar#1{\overline{#1}}

\def\ra{\rightarrow}
\def\into{\hookrightarrow}

\def\Mtwo{{\em Macaulay} 2\expandafter}
\def\Wedge{{\mathop \wedge}}

\numberwithin{equation}{section}
\begin{document}

\title{On the cohomology of a simple normal crossings divisor}
\author{Parsa Bakhtary}
\address{Department of Mathematics,
Purdue University,
150 North University Street,
West Lafayette, IN  47907-2067}
\begin{abstract}
We establish a formula which decomposes the cohomologies of various sheaves on a simple normal crossings divisor (SNC) 
$D$ in terms of the simplicial cohomologies of the dual complex $\D(D)$ with coefficients in a presheaf of vector spaces.  This presheaf consists precisely of the corresponding cohomology data on the components of $D$ and on their intersections. We use this formula to give a Hodge decomposition for SNC divisors and investigate the toric setting.  We also conjecture the existence of such a formula for 
effective non-reduced divisors with SNC support, and show that this would imply the vanishing of the higher simplicial cohomologies of the dual complex associated to a resolution of an isolated rational singularity.  
\end{abstract}
\maketitle

\section{Introduction}
The purpose of this paper is to establish and exploit a connection between the simplicial cohomology of topology and the Zariski sheaf 
cohomology of algebraic geometry, in order to understand the cohomological behavior of certain sheaves on a simple normal crossings 
(SNC) divisor.  To a simple normal crossings divisor $D=\sum D_i$ on a smooth variety, that is one such that each component $D_i$ is smooth and each $D_{i_1} \cap ... \cap D_{i_t}$ is a smooth transverse intersection, one may associate a dual CW complex $\D(D)$ using only 
incidence information, and under the assumption that each $D_{i_0}\cap ... \cap D_{i_p}$ is irreducible, $\D(D)$ is actually a simplicial complex.  Just think of the prime components $D_i$ as the vertices, non-empty 2-fold intersections $D_i \cap D_j$ as the edges, 
non-empty 3-fold intersections $D_i \cap D_j \cap D_k$ as the 2-faces, and so on.  This simplicial complex, first studied by G. L. Gordon \cite{G1}, is absent of any algebrao-geometric structure yet carries with it important skeletal information.

In a series of recent papers \cite{S1}, \cite{S2}, \cite{S3},
D. A. Stepanov investigated the homotopy type of the dual complex $\D(D)$ associated to a resolution of an isolated singularity (one 
exists by Hironaka's resolution theorem \cite{hiro}), which 
he proves is independent of the choice of resolution (using the weak factorization theorem of birational maps, proven by W{\l}odarczyk 
in \cite{W1} and \cite{W2}, and by Abramovich-Karu-Matsuki-W{\l}odarczyk in \cite{akmw}) and is therefore an invariant of the 
singularity.  It is known that the homotopy type of a rational surface singularity is trivial \cite{art}, 
and in this spirit Stepanov shows that $\D(D)$ is contractible, i.e. homotopic to a point, for isolated toric singularities, 
$3$-dimensional isolated rational hypersurface singularities, $3$-dimensional terminal singularities, and Brieskorn singularities.  He asks if the homotopy type of the dual complex associated to an isolated rational singularity is trivial in higher dimensions, and we study the connection between the simplicial cohomologies of $\D(D)$ and $H^i(D,\calO_D)$ with this conjecture in mind.

In fact, using a completely different approach involving Berkovich's analytic spaces, A. Thuillier has shown \cite{th}  more generally that given any ideal $\calI \subset \calO_X$ defining a closed subscheme $Z$ of a scheme $X$ over a perfect field 
and a proper map $f:X' \ra X$, with $f^{-1}Z=D$ a normal crossings divisor on $X'$ regular, restricting to an isomorphism 
$X'-D \cong X-Z$, then the homotopy type of the incidence complex $\D(D)$ depends only on $\calI$ and $X$.  In view of this result, given an arbitrary subscheme $Z \subset X$, the homotopy type of the exceptional divisor after taking an embedded resolution is independent of the choice of the resolution, and is an invariant the pair $(X,Z)$.

If $D=\sum D_i$ is a simple normal crossings (SNC) divisor on a compact K\"ahler manifold $X$ with each $D_{i_0...i_p}:=D_{i_0}\cap ... \cap D_{i_p}$ irreducible, then let $\D(D)$ denote the associated dual simplicial complex.  Also let $\tilde{\Omega}^r_D$ denote the  
sheaf of reduced holomorphic $r$-forms on $D$, i.e. 
\[\tilde{\Omega}^r_D=\{(\alpha_i) \in \bigoplus_i \Omega^r_{D_i} : {\alpha_i}_{|D_{ij}}={\alpha_j}_{|D_{ij}} , \forall i<j \}  \]
and we have that $\tilde{\Omega}^0_D \cong \calO_D$.  These forms are called reduced because they are precisely the K\"ahler $r$-forms on $D$ modulo torsion, i.e. modulo forms supported on the singular locus of $D$.  Now let $\calH_{dR}^q(\CC)$, $\calH^q(\Omega^r)$, and 
$\calH^q(\calO)$ denote the presheaves on $\D$ which assign to each $\D_{i_0...i_p}$ the vector space $H^q_{dR}(D_{i_0...i_p},\CC)$ (deRham cohomology), $H^q(D_{i_0...i_p},\Omega^r_{D_{i_0...i_p}})$, and $H^q(D_{i_0...i_p},\calO_{D_{i_0...i_p}})$ respectively.  

\begin{thm}
Suppose $X$ is a K\"ahler manifold (e.g. a smooth projective variety $/\CC$) and let $D=\sum D_i$ be a reduced SNC divisor on $X$ with 
each $D_{i_0}\cap ... \cap D_{i_p}$ irreducible, and denote by $\D=\D(D)$ the dual simplicial complex.  Then we have the following 
isomorphisms:

\[ H^i_{dR}(D,\CC) \cong \bigoplus_{p+q=i} H^p(\D,\calH^q_{dR}(\CC)) \]

\[ H^i(D,\tilde{\Omega}^r_D) \cong \bigoplus_{p+q=i} H^p(\D,\calH^q(\Omega^r)).\]

In particular, for $r=0$, we have

\[ H^i(D,\calO_D) \cong \bigoplus_{p+q=i} H^p(\D,\calH^q(\calO)).\]
\end{thm}

This formula shows that certain Zariski or deRham cohomology data on an SNC divisor $D$  can be constructed from (1) the dual simplicial incidence complex $\D(D)$, and (2) the corresponding cohomology data of the prime components $D_i$ and on their various intersections $D_{i_0...i_p}$.  The recipe is roughly to treat (2) as a presheaf on (1), take simplicial cohomology of $\D(D)$ with coefficients in that presheaf, then piece it together.  Deligne was the first \cite{del} to establish such a formula for the deRham cohomology of SNC divisors while developing his theory of mixed Hodge structures, and his proof uses the yoga of weights to establish the second page degeneration of the spectral sequence used in Section 4.  In \cite{del1}, there is a different proof of this degeneration using what is commonly referred to as the $\partial \bar \partial$-lemma.  Friedman gave a formula of this kind \cite{fr} for reduced K\"ahler differentials, and Stepanov gave one for the structure sheaf in \cite{S2}.  However, no concrete description of the components of each decomposition was given in these papers.  We unify these cohomological formulas in a combinatorial spirit, understanding the pieces of these decompositions as simplicial cohomologies with coefficients in a presheaf of vector spaces.  When that presheaf is $\calH^0(\calO) \cong \CC$, we see that the purely combinatorial cohomology $H^i(\D,\CC)$ lives inside $H^i(D,\calO_D)$.

It is worth remarking that this formula does not in any way depend on $D$ being an SNC {\em{divisor}} with irreducible component intersections.  In fact, we may take $D$ to be a suitable SNC cycle, i.e. a sum $D=\sum D_i$ where each component is a smooth irreducible $k_i$-dimensional subvariety of our ambient $n$-fold $X$, and all intersections $D_{i_0...i_p}:=D_{i_0}\cap ... \cap D_{i_p}$ are transversal, smooth, and irreducible of expected dimension $k_{i_0}+...+k_{i_p}-np$.  Geometrically speaking, an SNC cycle on an $n$-fold locally looks like a union of coordinate $k_i$-planes in $n$-space, and when each $k_i=n-1$, we recover the definition of an SNC divisor. 

\begin{cor}
In the above setting, we have a Hodge decomposition for $D$:

\[ H^i_{dR}(D,\CC) \cong \bigoplus_{p+q=i} H^p(D,\tilde{\Omega}^q_D). \]
\end{cor}

In the toric setting, the cohomology of the structure sheaf of an SNC divisor is purely combinatorial.
\begin{cor}
Let $X$ be a smooth complex projective toric variety, and let $D$ be a torus-invariant SNC divisor on $X$.  Then for every $i \geq 0$ we have an isomorphism 

\[ H^i(D,\calO_D) \cong H^i(\D,\CC). \] 
\end{cor}

Inspired by the ideas of \cite{S2}, we generalize and further develop the theory from the point of view of Stepanov, but note that 
similar avenues of thought have been explored by G. L. Gordon in the context of monodromy of complex analytic families \cite{G1}, by P. Deligne, P. Griffiths, J. Morgan, and D. Sullivan in the context of the cohomology of K\"ahler manifolds \cite{del1},  by R. Friedman in the context of deformations and smoothings of varieties with normal crossings \cite{fr}, and by P. Deligne \cite{del}, F. El Zein \cite{elz}, J. Carlson \cite{carl}, and J. H. M. Steenbrink \cite{St} in the context of mixed Hodge structures.  In Section 2, we briefly recall the notion of a presheaf on a simplicial complex and the cohomology of a simplicial complex with coefficients in a presheaf.  In Section 3, we establish some exact sequences that become our starting point in building a spectral sequence that we use in Section 4.  In Section 5 we state our main results and conjecture the existence of such a formula for effective non-reduced divisors with SNC support.  We then show that this would imply the vanishing of $H^i(\D(D),\CC)$ for $i>0$ when $D$ is the exceptional divisor of a resolution of an isolated rational singularity, and conclude with some examples.    

For the purposes of this paper, we shall always assume our divisor $D= \sum r_iD_i$ is such that each intersection $D_{i_0} \cap ... \cap D_{i_p}$ is irreducible.  By variety we mean an irreducible one, and we work over $\CC$.

\subsection{Acknowledgements}
The author is deeply grateful to his advisor J. W{\l}odarczyk for introducing him to this interesting subject and sincerely thanks him 
for his encouragement and patience.  The author also thanks D. Arapura for valuable suggestions regarding Dolbeault resolutions in the non-reduced setting, K. Matsuki for helpful comments on the examples, and J. McClure for an explanation of rationalization.

\section{Simplicial complexes and cohomology}

Let $\D$ be a finite simplicial complex, and fix a field $k$.

\begin{df}
A (covariant) \textsl{presheaf} $\calV$ of $k$-vector spaces on $\D$ consists of the following data:
\\
(a) for every subsimplex $\D'$ of $\D$, a $k$-vector space $\calV(\D')$, and
\\
(b) for every inclusion $\D'' \subset \D'$ of subsimplices of $\D$, a ``restriction'' morphism of $k$-vector spaces 
$r=r_{\D''\D'}: \calV(\D'') \rightarrow \calV(\D')$,
\\
\\
subject to the usual conditions
\\
\\
(0) $\calV(\varnothing)=0$,
\\
(1) $r_{\D'\D'}$ is the identity map on $\calV(\D')$,
\\
(2) if $\D''' \subset \D'' \subset \D'$ are three subsimplices of $\D$, then $r_{\D'''\D'}=r_{\D''\D'} \circ r_{\D'''\D''}$.
\end{df}

Thus our presheaf $\calV$ on $\D$ is a covariant functor from the category whose objects are subsimplices of $\D$ and whose morphisms 
are inclusions to the category of $k$-vector spaces.  This is slightly different from the usual definition of a presheaf on a 
topological space, which is usually contravariant. 

\begin{df}
We define the simplicial cohomology of $\D$ with coefficients in the presheaf $\calV$ using the \v{C}ech complex.

For each $p \geq 0$, define $C^p=C^p(\D,\calV)=\displaystyle \bigoplus_{i_0<...<i_p} \calV(\D_{i_0...i_p})$ where the direct sum is taken over all $p$-subsimplices $\D_{i_0...i_p}$ of $\D$.  For $v=(v_{i_0...i_p}) \in C^p$, the combinatorial \v{C}ech differential $\delta: C^p \rightarrow C^{p+1}$ is given by $(\delta v)_{i_0...i_{p+1}}= \displaystyle \sum_{k=0}^{p+1} (-1)^k {v_{i_0...\widehat{i_k}...i_{p+1}}}_{|\calV(\D_{i_0...i_{p+1}})} \in \calV(\D_{i_0...i_{p+1}})$, where the restriction map is  $r_{\D_{i_0...\widehat{i_k}...i_{p+1}}\D_{i_0...i_{p+1}}}$.  The compatibility of the restriction morphisms implies that $\delta$ is a well-defined $k$-linear map, and one checks that $\delta^2=0$.

We may now define the $p^{th}$ cohomology vector space of $\D$ with coefficients in $\calV$ to be $H^p(\D,\calV)=h^p(C^{\cdot}(\D,\calV))$.
\end{df}

\subsection{The Dual Complex of a Simple Normal Crossings Divisor}
Let $D=\sum D_i$ be a simple normal crossings (SNC) divisor on a complex projective variety $X$.  We define the dual complex $\D(D)$ to be the CW-complex whose cells are standard simplices $\D^k_{i_0...i_p}$ corresponding to the irreducible components $D^k_{i_0...i_p}$ of the nonempty intersections $D_{i_0}\cap ... \cap D_{i_p}= \bigcup_k D^k_{i_0...i_p}$.  That is to say, the 0-simplices (vertices) $\D_i$ of $\D(D)$ correspond to the prime components $D_i$ of $D$, the 1-simplices (edges) $\D_{ij}^k$ correspond to the irreducible components $D_{ij}^k$ of all nonempty intersections $D_i \cap D_j=\bigcup_kD_{ij}^k$, the 2-simplices (triangular faces) correspond to the irreducible components of nonempty triple intersections $D_i \cap D_j \cap D_k$, and so on.  The $p-1$ simplex $\D_{i_0...\widehat{i_s}...i_p}^j$ is a face of the $p$-simplex $\D_{i_0...i_p}^k$ iff $D_{i_0...i_p}^k \cap D_{i_0...\widehat{i_s}...i_p}^j \neq \varnothing$.  

This complex was first introduced by G. L. Gordon to study monodromy in analytic families \cite{G1}, and whose homotopy type has been 
more recently studied by D. A. Stepanov in the setting where $D$ is the exceptional divisor of a resolution of an isolated singularity 
\cite{S1}, \cite{S2}, \cite{S3}, and by A. Thuillier in a more general setting \cite{th}.

Notice that if dim $X = n$, then dim $\D(D) \leq n-1$, and if $D_{(r)}=\sum r_iD_i$ is an effective non-reduced divisor with SNC support, then $\D(D)=\D(D_{(r)})$, so there is no confusion in calling this complex $\D$.  We also restrict our attention to the case when $\D(D)$ is a simplicial complex, which happens iff each $D_{i_0...i_p}:=D_{i_0}\cap ... \cap D_{i_p}$ is irreducible. 

We consider certain presheaves of $\CC$-vector spaces over $\D(D)$.
\begin{ex}
The constant presheaf $\CC$ which assigns to $\CC$ to each $\D_{i_0...i_p}$, with identity restrictions.
\end{ex}
\begin{ex}
The presheaf $\calH^q(\calO_{(r)})$ which assigns $H^q(D^{(r)}_{i_0...i_p},\calO_{D^{(r)}_{i_0...i_p}})$ to each $\D_{i_0...i_p}$, where 
${D^{(r)}_{i_0...i_p}}$ is the scheme-theoretic intersection $r_{i_0}D_{i_0}\cap ... \cap r_{i_p}D_{i_p}$ defined in $X$ by the 
$\calO_X$-ideal $\calI^{r_{i_0}}_{i_0}+...+\calI^{r_{i_p}}_{i_p}$, with the natural $\CC$-linear maps on the $q^{th}$ cohomology level induced 
by the inclusion $D^{(r)}_{i_0...i_p}\into D^{(r)}_{i_0...\widehat{i_s}...i_p}$.  Here $\calI_i$ is the $\calO_X$-ideal defining $D_i$ in $X$.  
Note that if all $r_i=1$ then we are in the SNC case and the scheme-theoretic intersection $D_{i_0...i_p}$ coincides with the 
set-theoretic intersection, in which case we denote $\calH^q(\calO_{(1)})$ by $\calH^q(\calO)$.  Observe that $\calH^0(\calO) \cong \CC$ is the constant presheaf since we have natural isomorphisms 
$H^0(D_{i_0...i_p},\calO_{D_{i_0...i_p}}) \cong \CC$.
\end{ex}
\begin{ex}
The presheaf $\calH^q({\Omega}^r)$ which assigns 
$H^q(D_{i_0...i_p},\Omega^r_{D_{i_0...i_p}})$ to each $\D_{i_0...i_p}$, with the natural $\CC$-linear maps on the $q^{th}$ cohomology level 
induced by the inclusion $D_{i_0...i_p}\into D_{i_0...\widehat{i_s}...i_p}$.
\end{ex}

\begin{ex}
The presheaf $\calH^q_{dR}(\CC)$ which assigns the deRham (or singular) cohomology $H^q_{dR}(D_{i_0...i_p},\CC)$ to each $\D_{i_0...i_p}$, with the natural $\CC$-linear maps on the $q^{th}$ cohomology level induced by the inclusion $D_{i_0...i_p}\into D_{i_0...\widehat{i_s}...i_p}$
\end{ex}

Thus, using the definitions above, it makes sense to refer to $H^p(\D(D),\calV)$ where $\calV$ is any of the above presheaves.

\section{Some Technical Lemmas}

We first need to establish the exactness of a sheafified \v{C}ech complex resolving the structure sheaf of an effective divisor with simple normal crossing support.

Let $D_{(r)}=\displaystyle \sum_{i=1}^N r_iD_i$ be an effective divisor with SNC support on a smooth projective variety $X$ (for simplicity we say $D$ is NC), and as before let 
$D^{(r)}_{i_0...i_p}:=r_{i_0}D_{i_0}\cap ... \cap r_{i_p}D_{i_p}$ be the scheme-theoretic intersection in $X$.  That is to say, if 
$\calI_i \subset \calO_X$ is the ideal sheaf defining $D_i$ in $X$, then $D_{i_0...i_p}^{(r)}$ is defined by the ideal 
$\calI_{i_0}^{r_{i_0}} + ... + \calI_{i_p}^{r_{i_p}}$ in $\calO_X$. Notice that $D_{i}^{(r)}=r_iD_i$.  Let 
$\calC^p=\displaystyle \bigoplus_{i_0<...<i_p}\iota_*\calO_{D^{(r)}_{i_0...i_p}}$ where $\iota=\iota^{(r)}_{i_0...i_p}: D^{(r)}_{i_0...i_p} \into D_{(r)}$ is the 
natural inclusion.  We have a combinatorial \v{C}ech differential $\delta^p: \calC^p \rightarrow \calC^{p+1}$ defined as follows:

given $\alpha=(f_{i_0...i_p}) \in \calC^p(U)=\displaystyle \bigoplus_{i_0<...<i_p}\calO_{D^{(r)}_{i_0...i_p}}(U \cap D^{(r)}_{i_0...i_p})$ a section of the sheaf 
$\calC^p$ over an open set $U \subset D_{(r)}$, then
\[(\delta \alpha)_{i_0...i_{p+1}}(U)=\sum_{j=0}^{p+1}(-1)^j(f_{i_0...\widehat{i_j}...i_{p+1}})_{|U\cap D^{(r)}_{i_0...i_{p+1}}}\] so 
$(\delta \alpha)(U) \in \calC^{p+1}(U)$, and one easily checks that $\delta^{p+1} \circ \delta^p=0$ for any $p \geq 0$..  Also, there is a natural injection 
$\rho: \calO_{D_{(r)}} \into \calC^0=\displaystyle \bigoplus_i \iota_*\calO_{r_iD_i}$ given by restriction.  It is injective because if we are given a 
section $f \in \calO_{D_{(r)}}(U)$ over an affine open subset $U \subset D_{(r)}$, then $\rho_U(f)=0$ implies that $f_{|U\cap r_iD_i}=0$ for 
every $i$, hence $f$ is in the ideal $I_i^{r_i} \subset \calO_{D_{(r)}}(U)$ defining $r_iD_i$ in $D_{(r)}$ over $U$ for every $i$.  But then 
$f \in \bigcap I_i^{r_i}$ which means that $f$ must vanish on all of $D_{(r)} \cap U=\bigcup r_iD_i\cap U$, i.e. that $f=0$ on 
$D_{(r)}$ over $U$.
\begin{lem}
The complex described above 
\[0 \ra \calO_{D_{(r)}} \ra \calC^0 \overset{\delta^0} \ra \calC^1 \overset{\delta^1} \ra \calC^2 \overset{\delta^2} \ra ...\] 
is an exact sequence of $\calO_{D_{(r)}}$-modules.
\end{lem}

\begin{proof}
We have already shown that the first map is injective. We first treat the $p=0$ joint.  Suppose we are given a closed 0-cycle 
$\alpha=(f_i)$, i.e. $f_j-f_i=0$ on $D^{(r)}_{ij}$ for every $i<j$.  Replacing $D_{(r)}$ by an affine subset of $D_{(r)}$, we have 
$f_2-f_1 \in I^{(r)}_{12}$ the ideal of $D^{(r)}_{12}$ in $\calO_{D_{(r)}}$.  Since $D_{(r)}$ is NC, $I^{(r)}_{12}=I^{r_1}_1+I^{r_2}_2$ where 
$I_i$ is the ideal of $D_i$ in $\calO_{D_{(r)}}$.  Hence we can write $f_2-f_1=a_2-a_1$ where $a_i \in I^{r_i}_i$.  Now set 
$f^{(2)}=f_2-a_2=f_1-a_1 \in \calO_{D_{(r)}}$, which lifts both $f_1$ and $f_2$.  By assumption, we have $f^{(2)}-f_3=0$ on $D^{(r)}_{13}$ and 
$D^{(r)}_{23}$, which means 
\[f^{(2)}-f_3 \in I^{(r)}_{13} \cap I^{(r)}_{23}=(I_1^{r_1}+I_3^{r_3}) \cap (I_2^{r_2}+I_3^{r_3})=(I_1^{r_1} \cap I_2^{r_2}) +I_3^{r_3}.\]
All of these ideal equalities hold because $D_{(r)}$ is NC.  Write $f^{(2)}-f_3=a_{12}-a_3$ where $a_{12}\in I^{(r)}_{12}$ and 
$a_3\in I^{r_3}_3$, and set $f^{(3)}=f^{(2)}-a_{12}=f_3-a_3 \in \calO_{D_{(r)}}$, which is a lift of $f_1,f_2,$ and $f_3$.  Again by assumption 
we have $f^{(3)}-f_4=0$ on $D^{(r)}_{14}, D^{(r)}_{24}, D^{(r)}_{34}$, hence 
\[f^{(3)}-f_4 \in I^{(r)}_{14}\cap I^{(r)}_{24}\cap I^{(r)}_{34}=(I_1^{r_1}+I_4^{r_4})\cap (I_2^{r_2}+I_4^{r_4})\cap (I_3^{r_3}+I_4^{r_4})=(I_1^{r_1} \cap I^{r_2}_2 \cap I_3^{r_3})+I_4^{r_4}\]
Again these ideal equalities hold because $D_{(r)}$ is NC.  Hence we may write $f^{(3)}-f_4=a_{123}-a_4$ where $a_{123}\in I^{(r)}_{123}$ and 
$a_4\in I_4^{r_4}$.  Set $f^{(4)}=f^{(3)}-a_{123}=f_4-a_4 \in \calO_{D_{(r)}}$, which is a lift of $f_1,...,f_4$.  Continuing this way, we 
arrive at $f=f^{(N)} \in \calO_{D_{(r)}}$ lifting each $f_i$, and hence mapping to our cycle $(f_i)$.

For the general $p \geq 1$ case, we induct on the dimension of the ambient variety $X$.  If dim $X=0$, then there is nothing to show.
To show the inductive step, we induct on $N$, the number 
of components of $D_{(r)}$, the case $N=1$ being trivial, since our complex is simply 
\[0 \rightarrow \calO_{D_{(r)}} \rightarrow \calO_{D_{(r)}} \rightarrow 0\]
which is obviously exact.  Let $\alpha=(f_{i_0...i_p}) \in \calC^p$ be a closed $p$-cycle, and write 
\[\alpha=\alpha_{\neq 1}\oplus \alpha_1=(f_{i_0...i_p})_{i_0>1}\oplus (f_{1i_1...i_p})\]
The \v{C}ech differential of the complex associated to $r_2D_2+...+r_ND_N$ acts the same way as does the \v{C}ech differential for $D_{(r)}$ on 
the $\alpha_{\neq 1}$ component, hence the induction hypothesis on $N$ gives us a $\beta_{\neq 1}=(g_{i_0...i_{p-1}})_{i_0>1} \in 
\displaystyle \bigoplus_{1<i_0<...<i_{p-1}}\iota_* \calO_{D^{(r)}_{i_0...i_{p-1}}}$ such that $(\delta^{p-1}\beta_{\neq 1})_{i_0...i_p}=(\alpha_{\neq 1})_{i_0...i_p}$, $i_0>1$.

Now set $D_i'=D^{(r)}_{1i}=r_1D_1 \cap r_iD_i$ for $i=2,...,N$ so that $D'=\displaystyle \sum_{i=2}^N D_i'$ is an effective divisor on 
$D_1$ with SNC support.  
Denote by $D'_{i_1...i_p}=D'_{i_1} \cap ... \cap D'_{i_p}$ the scheme-theoretic intersection in $D_1$, and consider 
\[({g_{i_1...i_p}}_{|D^{(r)}_{1i_1...i_p}}-f_{1i_1...i_p}) \in \bigoplus_{1<i_1<...<i_p} \calO_{D'_{i_1...i_p}}=\bigoplus_{1<i_1<...<i_p} \calO_{D^{(r)}_{1i_i...i_p}}\]
as a $p-1$ cocycle in the complex associated to $D'$.  This cocycle is closed, i.e. we have that 
\[(g_{i_2...i_{p+1}}-f_{1i_2...i_{p+1}})_{|D^{(r)}_{1i_1...i_{p+1}}} - (g_{i_1i_3...i_{p+1}}-f_{1i_1i_3...i_{p+1}})_{|D^{(r)}_{1i_1...i_{p+1}}} + ... +(-1)^p(g_{i_1...i_p}-f_{1i_1...i_p})_{|D^{(r)}_{1i_1...i_{p+1}}}\]
\[=\sum_{j=1}^{p+1}(-1)^{j-1}{g_{i_1...\widehat{i_j}...i_{p+1}}}_{|D^{(r)}_{1i_1...i_{p+1}}} + \sum_{j=1}^{p+1}(-1)^j{f_{1i_1...\widehat{i_j}...i_{p+1}}}_{|D^{(r)}_{1i_1...i_{p+1}}}\]
\[={f_{i_1...i_{p+1}}}_{|D^{(r)}_{1i_1...i_{p+1}}} + \sum_{j=1}^{p+1}(-1)^j{f_{1i_1...\widehat{i_j}...i_{p+1}}}_{|D^{(r)}_{1i_1...i_{p+1}}}=0\]
by our assumption that $\alpha$ is a closed $p$-cycle and by our construction of the $g_{i_1...i_p}$.  Thus, by our induction hypothesis 
on dim $D_1 <$ dim $X$, this cocycle is a coboundary, meaning that there exists a $p-2$-cycle 
$(h_{i_1...i_{p-1}})=:(g_{1i_1...i_{p-1}})=:\beta_1\in \displaystyle \bigoplus_{1<i_1<...<i_{p-1}} \calO_{D'_{i_1...i_{p-1}}} = \displaystyle \bigoplus_{1<i_1<...<i_{p-1}} \calO_{D^{(r)}_{1i_1...i_{p-1}}}$ such that for every $1<i_1<...<i_p$, we have that 
\[(g_{1i_2...i_p} - g_{1i_1i_3...i_p} + ... + (-1)^{p-1}g_{1i_1...i_{p-1}})_{|D^{(r)}_{1i_1...i_p}} = {g_{i_1...i_p}}_{|D^{(r)}_{1i_1...i_p}} - f_{1i_1...i_p}.\]
Hence we may take $\beta=\beta_{\neq 1} \oplus \beta_1=(g_{i_0...i_{p-1}})_{i_0>1} \oplus (g_{1i_1...i_{p-1}})$ and one sees that by construction 
we have $\delta^{p-1}(\beta)=\alpha$ for every $i_0<...<i_p$, which completes the proof.
\end{proof}

We show briefly how the usual K\"ahler forms on a (reduced) SNC divisor $D=\sum D_i$ on $X$ are related to reduced forms, following Friedman \cite{fr} Section 1.  Since $D$ is reduced and a local complete intersection, the conormal sequence
\[ 0 \ra \calI/\calI^2 \ra {\Omega^1_X}_{|D} \ra \Omega^1_D \ra 0\]
is exact, and allows us to define the K\"ahler differentials on $D$, taking $\Omega^r_D:= \Wedge^r \Omega^1_D$.  There is a natural map 
$\Omega^r_D \ra \bigoplus \iota_* \Omega^r_{D_i}$ whose kernel $\tau^r_D$ consists of those differentials supported on Sing($D$).  We may then define $\tilde{\Omega}^r_D:= \Omega^r_D / \tau^r_D$.  One can directly show that this agrees with the definition given in the introduction.

\begin{lem} 
Let $D=\sum D_i$ be a SNC divisor, and let $\calE_r^p= \displaystyle \bigoplus_{i_0<...<i_p} \iota_* \Omega^r_{D_{i_0...i_p}}$.  Then we have an exact sequence of $\calO_D$-modules
\[0 \ra \tilde{\Omega}^r_D \ra \calE_r^0 \overset{\delta^0} \ra \calE_r^1 \overset{\delta^1} \ra \calE_r^2 \overset{\delta^2} \ra ...\] where the first map is inclusion and the other maps are the \v{C}ech differentials as in Lemma 3.1.
\end{lem}

\begin{proof}
R. Friedman gives a proof of this lemma in \cite{fr} Proposition 1.5 using a double induction scheme similar to ours in Lemma 3.1.  We sketch an alternate proof:  since $\tilde{\Omega}^r_D$ is defined to be $\ker \delta^0$ we get exactness at the $p=0$ joint for free.  For the remaining joints, note that this is a local question and choose local coordinates of $D_{i_0...i_{p+1}}$ such that we may extend them to local coordinates of $D_{i_0...i_p}$ and $D_{i_0...i_{p-1}}$.  Then noting that the various wedges $dz_{j_1} \wedge ... \wedge dz_{j_r}$ form a local basis of $\Omega^r_{D_{i_0...i_p}}$, a form $(\alpha) \in \calE_r^p$ being $\delta$-closed implies the corresponding cocycles of holomorphic coefficients of the basis elements are each $\delta$-closed as well.  We can lift these cocycles to coboundaries by Lemma 3.1, thus giving a natural choice of a coboundary form $(\beta) \in \calE_r^{p-1}$ which $\delta$ sends to $(\alpha)$.
\end{proof}

\begin{rem}
The above lemmas generalize in a straightforward way to give such exact sequences for an SNC {\em{cycle}} on $X$.
\end{rem}

\section{The Spectral Sequence}
We first work out in detail a spectral sequence associated to a SNC divisor $D= \sum D_i$, and later indicate the 
adjustments needed in the more general case where $D_{(r)}= \sum r_iD_i$ is an effective divisor with SNC support.  
This spectral sequence for the SNC case (or one similar to it) has been worked out in \cite{del}, \cite{del1}, \cite{G2}, and \cite{KK} for deRham cohomology, in \cite{S2} for the structure sheaf $\calO_D$, and is mentioned in \cite{fr} for reduced forms $\tilde{\Omega}^r_D$.  We follow Stepanov's presentation in \cite{S2}, but work with the sheaf of reduced holomorphic $r$-forms, which for $r=0$ includes the structure sheaf.  As usual we assume that each intersection $D_{i_0...i_p}$ is irreducible.  We have by Lemma 3.2 that

\[0 \ra \tilde{\Omega}^r_D \ra \calE_r^0 \overset{\delta^0} \ra \calE_r^1 \overset{\delta^1} \ra \calE_r^2 \overset{\delta^2} \ra ...\]
is exact, where 
$\calE_r^p=\displaystyle \bigoplus_{i_0<...<i_p}\iota_* \Omega^r_{D_{i_0...i_p}}$.  This implies that the complexes 
\[\Omega^*: \tilde{\Omega}^r_D \ra 0 \ra 0 \ra ...\] and 

\[\calE_r^*: \calE_r^0 \ra \calE_r^1 \ra \calE_r^2 \ra ...\] are quasiisomorphic.  Clearly the hypercohomologies of 
$\Omega^*$ are just the cohomologies of $\tilde{\Omega}^r_D $, i.e. $\HH^i(\Omega^*) \cong H^i(D,\tilde{\Omega}^r_D)$.  To calculate the hypercohomologies of $\calE_r^*$, we take Dolbeault resolutions 

\[0 \ra \calE_r^p \ra \calE_r^{p,0} \overset{\bar \partial} \ra \calE_r^{p,1} \overset{\bar \partial} \ra \calE_r^{p,2} \overset{\bar \partial} \ra ...\] where 
$\calE_r^{p,q}=\displaystyle \bigoplus_{i_0<...<i_p}\iota_* \calA^{r,q}_{D_{i_0...i_p}}$, $\calA^{r,q}_{D_{i_0...i_p}}$ is the sheaf of differential 
forms of type $(r,q)$ on $D_{i_0...i_p}$ and the maps are the Dolbeault differentials $\bar \partial$.

Now consider the bigraded collection of $\CC$-vector spaces
\[C^{p,q}:=H^0(D,\calE_r^{p,q})=\bigoplus_{i_0<...<i_p}H^0(D_{i_0...i_p},\calA^{r,q}_{D_{i_0...i_p}})=\bigoplus_{i_0<...<i_p} A^{r,q}_{D_{i_0...i_p}}\] equipped with the Dolbeault differential $\bar \partial: C^{p,q} \rightarrow C^{p,q+1}$ and the combinatorial \v{C}ech differential 
$\delta: C^{p,q} \rightarrow C^{p+1,q}$ defined as follows: given $\alpha=(\alpha_{i_0...i_p}) \in C^{p,q}$, we set 
\[(\delta \alpha)_{i_0...i_{p+1}}=\sum_{j=0}^{p+1}(-1)^{q+j}(\alpha_{i_0...\hat{i_j}...i_{p+1}})_{|D_{i_0...i_{p+1}}}.\]  
We have $\bar{\partial}\delta + \delta \bar{\partial}=0$, so that $(C^{*,*},\delta,\bar \partial)$ is a bicomplex.  
Take $(C^*,d)$ to be the total complex, i.e. $C^m=\displaystyle \bigoplus_{p+q=m}C^{p,q}$ and $d=\delta + \bar \partial$.  We may now express the hypercohomologies of $\calE_r^*$ as the cohomologies of $(C^*,d)$, that is to say $\HH^i(\calE_r^*) \cong h^i(C^*,d)$.

The filtration $F^pC^m=\displaystyle \bigoplus_{\overset{p'+q=m} {p'\geq p}}C^{p',q}$ on the complex $(C^*,d)$ gives rise to a spectral sequence $E_*$ that converges to the cohomologies $h^*(C^*,d)$ (see \cite{GH} Ch.3, section 5) and

\[E_0^{p,q} \cong C^{p,q};\]
\[E_1^{p,q} \cong H^q_{\bar \partial}(C^{p,*});\]
\[E_2^{p,q} \cong H_{\delta}^p(H^q_{\bar \partial}(C^{*,*})).\]

We have that \[E_1^{p,q} \cong \bigoplus_{i_0<...<i_p} H^q_{\bar \partial}(A^{r,*}_{D_{i_0...i_p}}) \cong \bigoplus_{i_0<...<i_p} H^q(D_{i_0...i_p},\Omega^r_{D_{i_0...i_p}}).\]

Observe that the cochain complex 
\[0 \ra E_1^{0,q} \ra E_1^{1,q} \ra E_1^{2,q} \ra ...\] 
with the induced \v{C}ech differentials $\delta$ is isomorphic to the cochain complex 
\[0 \ra \bigoplus_i H^q(D_i,\Omega^r_{D_i}) \ra  \bigoplus_{i_0<i_1} H^q(D_{i_0i_1},\Omega^r_{D_{i_0i_1}}) \ra 
\bigoplus_{i_0<i_1<i_2} H^q(D_{i_0i_1i_2},\Omega^r_{D_{i_0i_1i_2}}) \ra ...\] that one uses to calculate the cohomologies of $\D(D)$ with 
coefficients in the presheaf $\calH^q(\Omega^r)$.  Hence \[E_2^{p,q} \cong H_{\delta}^p(E_1^{*,q}) \cong H^p(\D(D),\calH^q(\Omega^r)).\]

We now show how to get such a spectral sequence for the structure sheaf of an effective non-reduced divisor $D_{(r)}=\sum r_iD_i$ with 
SNC support. 
Since we have our exact sequence from Lemma 3.1 in this setting, the only missing ingredient is a Dolbeault type resolution of the 
non-reduced structure sheaves $\calO_{D^{(r)}_{i_0...i_p}}$, where $D^{(r)}_{i_0...i_p}=r_{i_0}D_{i_0} \cap ... \cap r_{i_p}D_{i_p}$ is the 
scheme-theoretic intersection defined by $\calI:=\calI^{r_{i_0}}_{i_0} + ... + \calI^{r_{i_p}}_{i_p} \subset \calO_X$.  If we have this, we may then construct the bicomplex $(C^{*,*},\delta,\bar \partial)$ to obtain the spectral sequence as in the SNC case above. Consider the following diagram:

\begin{diagram}
0 & \rTo & \calI & \rTo & \calI \calA^{0,0}_X & \rTo^{\bar \partial} & \calI \calA^{0,1}_X & \rTo^{\bar \partial} & ...\\
  &      & \dTo  &      & \dTo              &                   & \dTo               &                   &    \\
0 & \rTo & \calO_X & \rTo & \calA^{0,0}_X      & \rTo^{\bar \partial} & \calA^{0,1}_X        & \rTo^{\bar \partial} & ...\\
  &      & \dTo  &      & \dTo              &                   & \dTo               &                   &    \\
0 & \rTo & {\calO_X}/{\calI} & \rTo & {\calA^{0,0}_X}/{\calI \calA^{0,0}_X} & \rTo^{\bar \partial} & {\calA^{0,1}_X}/{\calI \calA^{0,1}_X} & \rTo^{\bar \partial} & ...
\end{diagram}

The diagram commutes because our ideal is locally generated by holomorphic functions, which commute with $\bar \partial$, and it is 
easy to see that $\bar \partial (\calI \calA_X^{0,q}) \subset \calI \calA_X^{0,q+1}$.  We would like to show that the third row is exact, which would give us the acyclic resolution that we need.  Indeed, each ${\calA^{0,q}_X}/{\calI \calA^{0,q}_X}$ is a quotient of acyclic sheaves and is hence acyclic for the global section functor.  Since the columns are all short exact and the second row is the Dolbeault resolution of $\calO_X$ (and is thus exact), it suffices by the Snake Lemma to show that the first row is exact.  To do this we simply reprove the Dolbeault lemma while keeping track of the holomorphic coefficients in $\calI$.

\begin{lem}
Let $\calI \subset \calO_X$ be a finitely generated ideal sheaf, and let $f=f(z_1,...,z_n) \in \calI \calA^{0,0}_X(U)$ be a $C^{\infty}$ function on an open subset $U \subset X$ which is holomorphic in the variables $z_l$, $l>q$.  Then there exists an open subset $V \subset U$ and a $C^{\infty}$ function $g \in \calI \calA^{0,0}_X(V)$, holomorphic in the variables $z_l$, $l>q$, such that $\frac {\partial g} {\partial {\bar z_q}} = f$.
\end{lem}

\begin{proof}
Write $\calI(U)=I=(x_1,...,x_k)$ and $f=\sum x_jf_j$, where the sum runs over the generators of $I$, and $f_j$ are $C^{\infty}$.  It is well known (see \cite{V} Prop 2.32) that locally we can find $C^{\infty}$ functions $g_j$ such that $\frac {\partial g_j} {\partial {\bar z_q}} = f_j$, by taking 
$g_j= \frac {1} {2 \pi i} \int \frac{f_j(z_1,...,w_q,...,z_n)} {w_q - z_q} dw_q \wedge d{\bar w_q}$.  Now set $g=\sum x_jg_j$.  Since the $x_j$ are holomorphic, we have that \[\frac {\partial g} {\partial {\bar z_q}} = \sum x_j \frac {\partial g_j} {\partial {\bar z_q}} = \sum x_jf_j =f.\]  Also, $g$ remains holomorphic in the variables $z_l$, $l>q$ since we may pass the $x_j$ and the operator $\frac {\partial} {\partial {\bar z_l}}$ through the integral.
\end{proof}

\begin{lem}
The first row in the above diagram is exact.
\end{lem}

\begin{proof}
The proof is identical to that of the Dolbeault lemma (see \cite{V} Prop 2.31 and Prop 2.36) which uses a lemma similar to the one 
above, the only addition being that the coefficients of the forms in question are always in $\calI \calA^{0,0}_X$.
\end{proof}

Let $\calA^{0,q}_{D^{(r)}_{i_0...i_p}}$ denote ${\calA^{0,q}_X}/{\calI \calA^{0,q}_X}$, where $\calI:=\calI^{r_{i_0}}_{i_0} + ... + \calI^{r_{i_p}}_{i_p}$ is the $\calO_X$ ideal that defines the scheme-theoretic intersection $D^{(r)}_{i_0...i_p}=r_{i_0}D_{i_0} \cap ... \cap r_{i_p}D_{i_p}$ in $X$.  Also let $\calC^{p,q}=\displaystyle \bigoplus_{i_0<...i_p} \calA^{0,q}_{D^{(r)}_{i_0...i_p}}$ and let 
$C^{p,q}=\displaystyle \bigoplus_{i_0<...i_p} A^{0,q}_{D_{i_0...i_p}}$ denote the global sections.

Thus, we have our spectral sequence for the structure sheaf of an effective non-reduced divisor with SNC support. We now give a quick 
Hodge-theoretic proof of the second page degeneration of our spectral sequence in the reduced case, similar to the one given in 
\cite{del1}.

\begin{lem}[$\partial \bar \partial$-lemma]
Suppose $\omega$ is a global $(r,q)$ form on a K\"ahler manifold $D_{i_0...i_p}$ that is both $\partial$ and $\bar \partial$ closed.  If $\omega$ is $\bar \partial$ exact or $\partial$ exact, then $\omega= \partial \bar \partial \gamma$, for some form $\gamma$ of type $(r-1,q-1)$.
\end{lem}

\begin{proof}
See either \cite{del1} or \cite{V} Prop 6.17.  The idea is to use the decomposition of the space of forms into the kernel and image of the Laplacian, and employ the K\"ahler identities.
\end{proof}

Now let $K^{p,q}=\displaystyle \bigoplus_{i_0<...<i_p} \ker \partial \subset  \displaystyle \bigoplus_{i_0<...<i_p} A^{r,q}_{D_{i_0...i_p}}= C^{p,q}$, where $\partial: A^{r,q} \ra A^{r+1,q}$ is our usual operator and let 
$H^{p,q}=\displaystyle \bigoplus_{i_0<...<i_p} H^q(D_{i_0...i_p},\Omega^r_{D_{i_0...i_p}})$.  Consider the natural maps of double complexes 

\[(H^{*,*},\delta,0) \leftarrow  (K^{*,*},\delta,\bar \partial) \into (C^{*,*},\delta,\bar \partial)\]

where the first map is via the identification of $\partial$ and $\bar \partial$ cohomology on a K\"ahler manifold.  We claim that the induced maps on the first pages of the associated spectral sequences is an isomorphism.  Indeed, we saw above that the first page of the spectral sequence associated to $C^{*,*}$ is just $E^{p,q}_1 \cong \displaystyle \bigoplus_{i_0<...<i_p} H^q_{\bar \partial}(C^{p,*})=\displaystyle \bigoplus_{i_0<...<i_p} H^q(D_{i_0...i_p},\Omega^r_{D_{i_0...i_p}})$, so it suffices to show that the natural map 
\[{\ker \partial \cap \ker \bar \partial}/{\ker \partial \cap \im \bar \partial}= {\ker \partial \cap \ker \bar \partial}/{\im \partial \bar \partial}\ra H^q(D_{i_0...i_p},\Omega^r_{D_{i_0...i_p}}) \cong \ker \bar \partial / \im \bar \partial\] is an isomorphism for each $i_0<...<i_p$.  To see injectivity, given a $(r,q)$ form $\alpha \in \ker \partial \cap \ker \bar \partial$, it maps to $0$ in cohomology iff $\alpha$ is $\bar \partial$ exact.  Applying the $\partial \bar \partial$-lemma gives $\alpha= \partial \bar \partial \gamma$, hence $\alpha=0$.  For surjectivity, given a $\bar \partial$ closed form $\beta$, we may apply the $\partial \bar \partial$-lemma to $\partial \beta$ and conclude that $\partial \beta = \partial \bar \partial \gamma$.  Then $\beta - \bar \partial \gamma$ also represents $[\beta]$ and lies in $\ker \partial \cap \ker \bar \partial$, and hence maps to $[\beta]$.

So the induced map on first pages is an isomorphism, hence all higher pages are isomorphic by the Zeeman Comparison Theorem, see \cite{W} Theorem 5.2.12.  However, the total differential of $H^{*,*}$ is just $\delta$, so the corresponding spectral sequence must collapse at the second page, and hence so must the one corresponding to $C^{*,*}$.   

So our spectral sequence degenerates at the second page, i.e. $E_2=E_{\infty}$.  Our calculations give 
\[H^i(D,\tilde{\Omega}^r_D) \cong \HH^i(\Omega^*) \cong \HH^i(\calE_r^*) \cong h^i(C^*,d) \cong \bigoplus_{p+q=i}E_2^{p,q} \cong \bigoplus_{p+q=i}H^p(\D(D),\calH^q(\Omega^r)).\]

Taking $r=0$ we see that \[H^i(\D,\CC) \cong H^i(\D,\calH^0(\calO)) \cong E^{i,0}_2 \cong E^{i,0}_{\infty} \into H^i(D,\calO_D).\]

A very similar argument holds for the second page degeneration of the spectral sequence for the deRham cohomology of a SNC divisor on a K\"ahler manifold, giving the same kind of decomposition.  One shows that the sequence of $\calO_D$-modules
\[ 0 \ra \CC \ra \bigoplus_i \iota_* \CC \overset{\delta^0} \ra \bigoplus_{i_0<i_1} \iota_* \CC \overset{\delta^1} \ra \bigoplus_{i_0<i_1<i_2} \iota_* \CC \overset{\delta^2} \ra ... \]
is exact, resolves each $\CC$ on $D_{i_0...i_p}$ by the deRham resolution
\[ 0 \ra \CC \ra \calA^0_{D_{i_0...i_p}} \overset{d^0} \ra \calA^1_{D_{i_0...i_p}} \overset{d^1} \ra \calA^2_{D_{i_0...i_p}} \overset{d^2} \ra ... \]
and sets $C^{p,q}:= \displaystyle \bigoplus_{i_0<...<i_p} H^0(D_{i_0...i_p},\calA^q_{D_{i_0...i_p}})$.  Then $(C^{*,*},d,\delta)$ is a bicomplex whose total complex $(C^*,d + \delta)$ has the same cohomology as $H_{dR}^*(D,\CC)$.  The usual filtration on the bicomplex gives a spectral sequence as above, and one argues that this degenerates at the second page either by hand, using the Hodge decomposition and the $\partial \bar \partial$-lemma, or more compactly by considering three double complexes as above.  In the latter scenario, the middle bicomplex is associated to $\ker d^c$, where $d^c= \sqrt{-1}(\bar \partial - \partial)$, and the $dd^c$-lemma is used, which is equivalent to the $\partial \bar \partial$-lemma.  For details of either argument, see \cite{KK} or \cite{del1} respectively.    

\newpage

\section{Main Results, Conjectures, and Examples}

Applying the spectral sequence discussed in the previous section to the exact sequences from the lemmas in Section 3 yields formulas,
 which are essentially a reformulation of Deligne's result for deRham cohomology and Stepanov's result for the structure sheaf cohomology of a SNC divsor in the language of cohomologies of presheaves on the associated dual simplicial complex.  We summarize them now.  

\begin{thm}
Suppose $X$ is a K\"ahler manifold and let $D=\sum D_i$ be a reduced SNC divisor on $X$ with 
each $D_{i_0}\cap ... \cap D_{i_p}$ irreducible, and denote by $\D=\D(D)$ the dual simplicial complex.  Then we have the following 
isomorphisms:

\[ H^i_{dR}(D,\CC) \cong \bigoplus_{p+q=i} H^p(\D,\calH^q_{dR}(\CC)) \]

\[ H^i(D,\tilde{\Omega}^r_D) \cong \bigoplus_{p+q=i} H^p(\D,\calH^q(\Omega^r)).\]

In particular, for $r=0$, we have

\[ H^i(D,\calO_D) \cong \bigoplus_{p+q=i} H^p(\D,\calH^q(\calO)).\]

\end{thm}

These formulas are interesting in their own right because they show that no new cohomological information is created when smooth 
divisors are added to one another in a normal crossings fashion, provided we know how they intersect 
via the dual incidence complex $\D(D)$.  We also recover a Hodge decomposition for a SNC divisor on a K\"ahler manifold.

\begin{cor}
In the setting above, we have an isomorphism:

\[ H^i_{dR}(D,\CC) \cong \bigoplus_{r+q=i} H^q(D,\tilde{\Omega}^r_D) \]

\end{cor}  

\begin{proof}
The Hodge decomposition for each $D_{i_0...i_p}$ gives natural isomorphisms 
\[ H^q_{dR}(D_{i_0...i_p},\CC) \cong \bigoplus_{r+s=q} H^s(D_{i_0...i_p},\Omega^r_{D_{i_0...i_p}})\]
for every $i_0<...<i_p$ so we have isomorphisms $\calH_{dR}^q(\CC) \cong \displaystyle \bigoplus_{r+s=q} \calH^s(\Omega^r)$ as presheaves 
on $\D$.  Using Theorem 5.1 above we have a chain of isomorphisms

\[ H^i_{dR}(D,\CC) \cong \bigoplus_{p+q=i} H^p(\D,\calH^q_{dR}(\CC)) \cong \bigoplus_{p+q=i} H^p(\D,\bigoplus_{r+s=q} \calH^s(\Omega^r)) 
\cong \bigoplus_{p+r+s=i} H^p(\D, \calH^s(\Omega^r)) \] \[\cong \bigoplus_{r=0}^i \bigoplus_{p+s=i-r} H^p(\D, \calH^s(\Omega^r)) \cong 
\bigoplus_{r=0}^i H^{i-r}(D,\tilde{\Omega}^r_D) \cong \bigoplus_{r+q=i} H^q(D,\tilde{\Omega}^r_D). \]
\end{proof}

We make the following conjecture for the more general non-reduced case.

\begin{conj}
In the setting of Theorem 5.1, let $D_{(r)}=\sum r_iD_i$ be an effective non-reduced divisor with SNC support on $X$, assume 
each $D_{i_0}\cap ... \cap D_{i_p}$ irreducible, and denote by $\D=\D(D_{(r)})$ the dual simplicial complex.  Then the spectral sequence 
constructed in Section 4 for the structure sheaf $\calO_{D_{(r)}}$ degenerates at the second page, yielding the isomorphism

\[ H^i(D,\calO_{D_{(r)}}) \cong \bigoplus_{p+q=i} H^p(\D,\calH^q(\calO_{(r)})).\]
\end{conj}

This seems quite hard, since we do not have the $\partial$ operator and we cannot do Hodge theory.  One encounters the same problem trying to give such a formula for an arbitrary vector bundle on a (reduced) SNC divisor $D$.

\subsection{The Toric Setting}
In the setting where $X$ is a smooth projective toric variety $/\CC$ and $D$ is a torus-invariant SNC divisor, 
our formula simplifies greatly, and we see that the Zariski cohomology of the structure sheaf $\calO_D$ is precisely the simplicial 
cohomology of $\D(D)$.

\begin{cor}
In the toric setting above, we have an isomorphism 
\[H^i(D,\calO_D) \cong H^i(\D(D),\CC).\]
\end{cor}

\begin{proof}
It is well-known that for a globally generated line bundle on a simplicial complete toric variety, the higher cohomology groups vanish 
(see \cite{ful} Section 3.5).  Thus the presheafs $\calH^q(\calO)$ are $0$ for $q>0$ since $H^q(D_{i_0...i_p},\calO_{D_{i_0...i_p}})=0$ 
for $q>0$, and we are done by Theorem 5.1 since we saw that $\calH^0(\calO) \cong \CC$ in Example 2.4.
\end{proof}

\begin{rem}
More generally, on a smooth projective complex variety $X$ (not necessarily toric) suppose that we have an SNC divisor $D$ such that 
$H^q(D_{i_0...i_p},\calO_{D_{i_0,...i_p}})=0$ for every $q>0$ and each $i_0<...<i_p$.  Then Theorem 5.1 again immediately gives that 
\[H^i(D,\calO_D) \cong H^i(\D(D),\CC).\]
Hence invariants such as the Euler characteristic coincide, i.e. $\chi(\calO_D) = \chi(\D(D))$.
\end{rem}

\newpage

\subsection{Rational Singularities}
Our starting point in this section is the following result of Stepanov, generalized to arbitrary singular loci over a perfect field by 
A. Thuillier in \cite{th}.

\begin{thm}[Stepanov]
The homotopy type of the dual complex $\D(D)$ of an SNC divisor $D$ associated to a resolution of an isolated singularity $o \in Y$ is 
independent of the choice of the resolution.  That is to say, the homotopy type of $\D(D)$ depends only on the isolated singularity 
$o$ and the ambient variety $Y$.
\end{thm}

\begin{proof}
The proof uses the weak factorization theorem of birational maps to reduce to the case of a blowup.  See \cite{S1} for details.
\end{proof}

Notice that by taking additional blowups, we can always ensure the irreducibility of each intersection $D_{i_0...i_p}$ without affecting the homotopy type of $\D(D)$.  Such a resolution is called {\em{good}}.  As we mentioned in the introduction, Stepanov has shown the contractibility of $\D$ for various kinds of isolated singularities: toric, Brieskorn, $3$-dimensional terminal, $3$-dimensional rational hypersurface; see \cite{S1}, \cite{S2}, and \cite{S3} for details.  Stepanov asks if $\D(D)$ is contractible for all isolated rational singularities of dimension $\geq 3$.  We conjecture a weaker, but more plausible assertion.

\begin{conj}
Suppose $f: X \rightarrow Y$ is a good resolution of an isolated rational singularity $o \in Y$ so that the exceptional divisor 
$D_{(r)}=\sum r_iD_i$ is effective with SNC support.  Then $H^i(\D,\CC)=0$ for all $i>0$.
\end{conj}

We show that Conjecture 5.3 implies Conjecture 5.7. For any $i_0<...<i_p$ we have a short exact sequence 
(as coherent sheaves on $D_{i_0...i_p}^{(r)}$)
\[ 0 \ra \calI_{i_0...i_p}^{(r)} \ra \calO_{D_{i_0...i_p}^{(r)}} \ra \calO_{D_{i_0...i_p}} \ra 0\]

where $\calI_{i_0...i_p}^{(r)} = \displaystyle \frac {(\calI_{i_0} + ... + \calI_{i_p})} {(\calI_{i_0}^{r_{i_0}} + ... + \calI_{i_p}^{r_{i_p}})}$ is 
the ideal defining $D_{i_0...i_p}$ in $D_{i_0...i_p}^{(r)}$.  Taking global sections we get another short exact sequence

\[ 0 \ra H^0(\calI_{i_0...i_p}^{(r)}) \ra H^0(\calO_{D_{i_0...i_p}^{(r)}}) \ra \CC \ra 0\]

where the last restriction map is surjective because the middle term is a nonzero complex vector space (it contains the constant functions) and the second map is restriction.  But a short exact sequence of vector spaces splits, hence we have

\[ H^0(\calO_{D_{i_0...i_p}^{(r)}}) \cong \CC \oplus H^0(\calI_{i_0...i_p}^{(r)}) \]

Since we get such natural isomorphisms for every $i_0<...<i_p$, we conclude that

\[ \calH^0(\calO_{D_{(r)}}) \cong \CC \oplus \calH^0(\calI^{(r)}) \]

as presheaves on $\D$, where $\calH^0(\calI^{(r)})(\D_{i_0...i_p}) = H^0(D_{i_0...i_p}^{(r)},\calI_{i_0...i_p}^{(r)})$.  It follows that we have 
an isomorphism on the level of cohomology

\[ H^i(\D,\calH^0(\calO_{D_{(r)}})) \cong H^i(\D,\CC) \oplus H^i(\D,\calH^0(\calI^{(r)})) \]

so that we have inclusions

\[ H^i(\D,\CC) \into H^i(\D,\calH^0(\calO_{(r)})) \into H^i(D_{(r)},\calO_{D_{(r)}}) \cong (R^if_*\calO_X)_o \]

and the last vector space is $0$ since our singularity was assumed to be rational.

\begin{rem}
Notice that we do not need the entirety of Conjecture 5.3 to prove Conjecture 5.7 in the argument given above.  It suffices to have an 
inclusion \[ H^i(\D,\CC) \into H^i(D_{(r)}, \calO_{D_{(r)}}) \cong (R^if_*\calO_X)_o \]
of the combinatorial part of the cohomology into the cohomology of the structure sheaf of $D_{(r)}$.
\end{rem}

We state the following interesting result of Stepanov.
\begin{thm}[Stepanov]
If $\D=\D(D)$ is the dual complex associated to a good resolution of an isolated hypersurface singularity of dimension $\geq 3$, 
then $\pi_1(\D)=0$, i.e $\D$ is simply connected.
\end{thm}

\begin{proof}
See \cite{S2}, Theorem 3.1.  The idea is to consider the link $M$ of the isolated singularity, which is simply connected, and think of 
$M$ as the border of a tubular neighborhood of the exceptional divisor $D$, which gives a surjective map $\varphi: M \ra D$ whose fibers are tori.  One concludes that $\pi_1(D)=0$ and then argues that this implies that $\pi_1(\D(D))=0$.
\end{proof}

Unfortunately, even if Conjecture 5.7 were true, we still would not know the cohomology of $\D$ with coefficients in $\ZZ$ so we could not conclude that $\D$ is contractible in the case of an isolated rational hypersurface singularity of dimension $\geq 3$. But, we could say this of the rationalization of $\D$.  Roughly speaking, the rationalization of a simply connected CW complex $\D$ is a canonically constructed simply connected CW complex $\D_0$, unique up to homotopy, with a map $\D \ra \D_0$ that induces isomorphisms on the level of homology and homotopy groups after tensoring those of $\D$ with $\QQ$.  See \cite{GM} for a discussion of localization and rationalization of a simply connected CW-complex.

\begin{conj}
If $\D=\D(D)$ is the dual complex associated to a resolution of an isolated rational hypersurface singularity of dimension 
$\geq 3$, then the rationalization of $\D$ is contractible.
\end{conj}

We have that $\pi_1(\D)=0$ from the above theorem.  Conjecture 5.7 would give $H^i(\D,\CC)=0$ for $i>0$, so the result would follow from the Hurewicz and Whitehead theorems, and the theory of localization.

\subsection{Examples}
We finish with some examples to illustrate the theory.

\begin{ex}
Let $X=\PP^n_{\CC}$ be complex projective $n$-space and let $D=\displaystyle \sum_{i=1}^{n+1} H_i$, where the $H_i$ are the coordinate hyperplanes.  Then $\D(D) \approx S^{n-1}$, and we know that $H^q(D_{i_0...i_p}, \calO_{D_{i_0...i_p}})$ is isomorphic to $\CC$ when $q=0$ and is $0$ for $q>0$, since $D_{i_0...i_p} \cong \PP^{n-1-p}$.  Then Theorem 5.1 (or Corollary 5.4) implies that $H^i(D,\calO_D) \cong H^i(S^{n-1},\CC)$, which we know is $\CC$ for $i=0,n-1$ and $0$ else.  
\end{ex}

\begin{ex}
Building on the previous example, let $X=X(\Sigma)$ be a smooth complex projective toric variety of dimension $n \geq 2$, and let $D=\sum D_i$ be the SNC sum of all the torus invariant divisors corresponding to edges of the fan $\Sigma$, so that $\omega_X \cong \calO_X(-D)$.  By Corollary 5.4 we have $H^0(\D(D),\CC) \cong H^0(D,\calO_D) \cong \CC$ and we can compute the higher cohomologies of $\calO_D$ by taking cohomology of the short exact sequence 
\[ 0 \ra \omega_X \ra \calO_X \ra \calO_D \ra 0 \]
and using that $H^i(X,\calO_X)=0$ for $i>0$ to conclude $H^i(\D(D),\CC) \cong H^i(D,\calO_D) \cong H^{i+1}(X,\omega_X) \cong H^{n-1-i}(X,\calO_X)$ for $i>0$, which is precisely $\CC$ when $i=n-1$ and $0$ else, confirming the observation that $\D(D) \approx S^{n-1}$ here as well.
\end{ex}

\begin{ex}[Euler characteristic]
A consequence of the exact sequence of Lemma 3.1 on the level of Euler characteristics is

\[ \chi(\calO_D)= \sum_{p \geq 0}(-1)^p (\sum_{i_0<...<i_p} \chi(\calO_{D_{i_0...i_p}})) \]
and we can calculate each $\chi(\calO_{D_{i_0...i_p}}))$ from the Hirzebruch-Riemann-Roch theorem.  For instance, if 
$C=\displaystyle \sum_{i=1}^N C_i$ is a SNC curve on a surface $X$ with each $C_i \cap C_j$ either empty or a single point for $i<j$, 
then the Riemann-Roch theorem gives us that 
\[ \chi(\calO_{C_i}) = 1 - g(C_i). \]
Letting $e:= |C_i \cap C_j \neq \varnothing|$ be the number of edges of $\D=\D(C)$ and noting $N$ is the number of vertices of $\D$, 
our formula above becomes
\[ \chi(\calO_C)= \sum_{i=1}^N \chi(\calO_{C_i}) - \sum_{j=1}^e \chi(\calO_{pt}) = N - e - \sum_{i=1}^N g(C_i) = \chi(\D) - \sum_{i=1}^N g(C_i) \]
which is a natural generalization to the SNC case. 

\end{ex}

\end{document}